%% file: main.tex
\documentclass[letterpaper, twocolumn]{article}
\usepackage{spconf,amsmath,amssymb,graphicx}
\usepackage[none]{hyphenat}
\usepackage{booktabs}
\usepackage[table,xcdraw]{xcolor}
\usepackage{rotating}
\usepackage{array}
\usepackage{makecell}
\usepackage{adjustbox}
\usepackage{pifont}
\usepackage{placeins}
\usepackage[makeroom]{cancel}
\usepackage[ruled,vlined]{algorithm2e}
\newcommand{\cmark}{{\color[HTML]{1ac938}\ding{51}}}%
\newcommand{\xmark}{{\color[HTML]{e8000b}\ding{55}}}%
\usepackage{hyperref}
\hypersetup{
  colorlinks   = true, 
  urlcolor     = black, 
  linkcolor    = black, 
  citecolor   = black 
}

\newcommand{\qedsymbol}{\hfill \ensuremath{\square}}

\usepackage{amsthm}

\newtheorem{theorem}{Theorem}
\newcommand{\theoremnewline}{\makebox[0pt][l]{}\\}

\renewenvironment{proof}{{\noindent \bfseries Proof:\newline}}{\newline}

\renewcommand{\subsubsection}[1]{\vspace{0.5\baselineskip}\noindent{\bf #1.}}
\renewcommand{\paragraph}[1]{\vspace{0.5\baselineskip}\noindent\textit{#1.}}

\title{A Comparison of Sparse Solvers for Severely Ill-Conditioned Linear Systems in Geophysical Marker-In-Cell Simulations}
\name{\shortstack{\textit{Marcel Ferrari} \\ mferrari@ethz.ch}}
\address{Faculty of Computational Science and Engineering, \\ Department of Mathematics, ETH Z\"urich\\Z\"urich, Switzerland}

\begin{document}

\maketitle

\begin{abstract}
Solving sparse linear systems is a critical challenge in many scientific and engineering fields, particularly when these systems are severely ill-conditioned. This work aims to provide a comprehensive comparison of various solvers designed for such problems, offering valuable insights and guidance for domain scientists and researchers. We develop the tools required to accurately evaluate the performance and correctness of 16 solvers from 11 state-of-the-art numerical libraries, focusing on their effectiveness in handling ill-conditioned matrices. The solvers were tested on linear systems arising from a coupled hydro-mechanical marker-in-cell geophysical simulation. To address the challenge of computing accurate error bounds on the solution, we introduce the Projected Adam method, a novel algorithm to efficiently compute the $\mathcal{L}_2$ condition number of a matrix without relying on eigenvalues or singular values. Our benchmark results demonstrate that Intel oneAPI MKL PARDISO, UMFPACK, and MUMPS are the most reliable solvers for the tested scenarios. This work serves as a resource for selecting appropriate solvers, understanding the impact of condition numbers, and improving the robustness of numerical solutions in practical applications.
\end{abstract}

\section{Introduction}\label{sec:intro}
Solving sparse linear systems (SLSs) is a crucial step in many scientific and engineering problems, particularly in the numerical solution of partial differential equations (PDEs). These systems arise naturally during the discretization process~\cite{pde_discretization_sparse_ls}, especially when addressing systems with multiple non-linearly interacting physical phenomena. For example, in computational fluid dynamics, they are essential for simulating fluid flow~\cite{fluid_fem_sparse_ls, fast_wave_sparse_ls} and turbulence~\cite{fem_fluid_1,fem_fluid_2,fem_fluid_3}. Similarly, in structural engineering, SLSs must be solved to model stress and strain distributions in complex structures~\cite{struct_fem_book, applications_fem_struct_eng}. In geophysics, understanding properties such as pressure, fluid flow, heat transfer, and mechanical deformation often involves solving large systems of equations~\cite{zilio_hydro_mechanical, zilio_subduction_earthquake}.

Problems arising from physics frequently result in severely ill-conditioned sparse matrices with high condition numbers, which significantly impact solver performance and accuracy. This issue is particularly prevalent in geophysics, where a linear system must encode physical properties of vastly different magnitudes. Consider, for example, the difference in magnitude between the velocities of moving plates, which typically range in the order of \mbox{$2-8 \ \operatorname{cm/yr} \approx 6\times 10^{-10} - 2.5\times 10^{-9} \ \operatorname{m/s}$}~\cite{plate_speed}, versus the internal deviatoric stresses, which normally range in the order of \mbox{$200-300 \ \operatorname{bar} = 2\times10^7 - 3 \times 10^7 \ \operatorname{N/m^2}$} but can exceed $1 \operatorname{kbar} = 10^8 \operatorname{N/m^2}$, such as in the volcanic regions of Hawaii~\cite{stress_magnitude}. These severely ill-conditioned SLSs can exhibit condition numbers greater than $10^{25}$, presenting a significant challenge even for robust and well established numerical libraries. 
This work aims to assess and compare the accuracy and performance of various high-performance, state-of-the-art sparse solvers. Using the 2D seismo-hydro-mechanical simulation developed in~\cite{Gerya2024} as a benchmark, we identify the solvers that handle ill-conditioned systems most efficiently. The evaluation includes both iterative and direct methods, optimized for both x86 CPUs and Nvidia GPU architectures. In addition to benchmarking the solvers, this work also develops the tools required to properly evaluate the correctness and accuracy of solutions – a task that is not widely recognized but is critical in the development of numerical code for scientific applications. We assess the capabilities and limitations of each solver and provide a comprehensive analysis in terms of accuracy and performance of each implementation. Moreover, this work provides a comprehensive list of available numerical libraries with extensive references, serving as a valuable resource to guide researchers and domain scientists facing comparable problems. The results highlight the key strengths and weaknesses of each method and offer a valuable starting point when optimizing similar workloads.

\section{Related Work}
Given the wide applications of sparse linear solvers, studies comparing their performance and accuracy are proposed periodically. Recent work in this area includes a study on HPC solvers for sparse linear systems with applications to PDEs, which discusses how the numerical solver is affected by the regularity of the discretization of the domain, the discretization of the input PDE and the choice of different initial conditions~\cite{hpc_solvers_comparison}. Another recent study focuses specifically on direct solvers applied to similarly ill-conditioned electromagnetic geophysical problems~\cite{evaluation_parallel_solvers}. Motivated by the need to explore a broader range of methods and algorithms, our work aims to provide a comprehensive comparison of solvers with a focus on ill-conditioned geophysical problems.

\section{Types of Sparse Linear System Solvers}
Over the years, many strategies have been developed for the efficient solution of sparse linear systems. Despite this, it is possible to divide them broadly into two distinct categories: iterative solvers and direct solvers. This section discusses the characteristics and properties of these two families of methods at a high level. The goal is to provide a broad overview of the techniques currently employed in state-of-the-art numerical libraries in order to better underscore the differences among the evaluated solvers. For further details on these techniques, please refer to the resources cited in the upcoming sections.

\subsection{Iterative Solvers}
Iterative solvers are methods that approach the solution of a linear system
\begin{equation}
\mathbf{A}\mathbf{x} = \mathbf{b}
\end{equation}
by gradually improving an initial guess through successive approximations. The core idea is to start with an initial guess $\mathbf{x}_0$, normally chosen at random or set to $\mathbf{0}$, and produce a sequence of approximations
\begin{equation}
\mathbf{x}_0, \mathbf{x}_1, \mathbf{x}_2, \mathbf{x}_3, \ldots
\end{equation}
that converge to the true solution $\mathbf{x}$. Each iteration involves applying an operator $\mathbf{\Phi}$ that moves the current approximation closer to the solution. This means that the previous sequence is equivalent to:
\begin{equation}
\mathbf{x}_0, \mathbf{\Phi x}_0, \mathbf{\Phi}^2 \mathbf{x}_0, \mathbf{\Phi}^3 \mathbf{x}_0, \ldots
\end{equation}
with $\mathbf{\Phi}$ chosen such that:
\begin{equation}
\lim_{n \rightarrow \infty} \mathbf{\Phi}^n \mathbf{x}_0 = \mathbf{x}
\end{equation}
Common iterative methods include the Jacobi method, Gauss-Seidel method, and Krylov subspace methods such as Conjugate Gradient (CG) and Generalized Minimal Residual (GMRES). An in-depth discussion of the most widely used iterative methods for sparse linear systems is available in~\cite{yousef_saad_iterative_methods_linear_systems}.

\subsubsection{Complexity}
The complexity of iterative solvers depends on the number of iterations required for convergence and the cost per iteration. The operator $\mathbf{\Phi}$ usually leverages matrix vector products in order to produce the next entry in the sequence. For dense systems, this has a complexity of $O(n^2)$, while for large sparse systems the cost per iteration depends on the number of non-zero entries $\operatorname{nnz}$ and is typically $O(\operatorname{nnz})$. The total computational cost is $O(k\cdot \operatorname{nzz})$, where $k$ is the number of iterations needed to achieve the desired accuracy.

\subsubsection{Convergence and Stability}
Convergence and stability of iterative solvers depend on the chosen method and the structural and numerical properties of the matrix $\mathbf{A}$. Properties such as symmetry and positive-definiteness can improve both stability and convergence speed. On the other hand, convergence can be slow if the matrix $\mathbf{A}$ is ill-conditioned and in extreme cases the algorithm can even diverge.

\subsubsection{Preconditioning}
Preconditioning is a technique used to improve the convergence rate of iterative solvers by transforming the linear system into an equivalent one with more favorable properties. That is, compute $\mathbf{P}$ and/or $\mathbf{Q}$ such that the \textbf{left}-preconditioned system
\begin{equation}
\mathbf{P}^{-1}\mathbf{A}\mathbf{x} = \mathbf{P}^{-1}\mathbf{b},
\end{equation}
the \textbf{right}-preconditioned system
\begin{equation}
\mathbf{A}\mathbf{Q}^{-1}(\mathbf{Q} \mathbf{x}) = \mathbf{b},
\end{equation}
or the \textbf{left}-\textbf{right}-preconditioned system
\begin{equation}
\mathbf{P}^{-1}\mathbf{A}\mathbf{Q}^{-1}(\mathbf{Q} \mathbf{x}) = \mathbf{P}^{-1}\mathbf{b}
\end{equation}
has better spectral properties, leading to faster convergence.

\subsubsection{Diagonal Preconditioning} 
Diagonal preconditioning involves scaling the matrix $\mathbf{A}$ by its diagonal elements. This is a cheap and simple preconditioning strategy which can be surprisingly effective, especially for diagonally dominant sparse systems~\cite{diag_preconditioning}.

\subsubsection{Incomplete LU Preconditioning}
Incomplete LU (ILU) preconditioning approximates $\mathbf{A}$ by factors $\mathbf{L}$ and $\mathbf{U}$, where $\mathbf{A} \approx \mathbf{LU}$, but with some entries in $\mathbf{L}$ and $\mathbf{U}$ discarded to maintain sparsity. Solving the system with this approximate factorization yields an approximate, yet still incorrect, solution so the factorization is instead used as a preconditioner for an iterative algorithm by setting $\mathbf{P/Q} = \mathbf{LU}$. ILU can significantly accelerate convergence for a wide range of problems, though it requires additional computational effort to construct the preconditioner
\cite{incomplete_LU_preconditioner}.

\subsubsection{Strengths}
Iterative methods have several advantages. They are scalable due to their reliance on simple matrix-vector products, which can be parallelized effectively. They are flexible, easily adaptable to different problems, and can incorporate various preconditioning techniques. Additionally, iterative solvers have lower memory requirements, as they usually only need to store a few vectors of the matrix size and also allow for matrix-free representations of the system, making them suitable for very large problems.

\subsubsection{Weaknesses}
Iterative solvers also have some weaknesses. They may converge slowly or fail to converge for ill-conditioned systems, such as the ones benchmarked in this work. Their performance and convergence are heavily dependent on the choice of preconditioner, making them sensitive to preconditioning.

\subsection{Direct Solvers}
Direct solvers operate aim to solve the linear system \begin{equation}
\mathbf{A}\mathbf{x} = \mathbf{b}
\end{equation}
through a finite sequence of operations, typically involving matrix factorizations. The core idea is to factorize the matrix $\mathbf{A}$ into products of simpler matrices, solve the resulting simpler systems, and then combine the solutions. Direct solvers guarantee a solution in a finite number of steps, assuming no numerical breakdown occurs. Common factorizations include Gaussian elimination ( also known as right-looking LU factorization), LU factorization, and QR factorization. An in-depth discussion of the most
widely used direct methods for sparse linear systems is
available in~\cite{tim_davis_direct_methods_linear_systems, survey_direct_solvers}.

\subsubsection{LU Factorization}
LU factorization decomposes the matrix $\mathbf{A}$ into a lower triangular matrix $\mathbf{L}$ and an upper triangular matrix $\mathbf{U}$ such that:
\begin{equation}
\mathbf{A} = \mathbf{LU}
\end{equation}
Once $\mathbf{L}$ and $\mathbf{U}$ are determined, the system $\mathbf{A}\mathbf{x} = \mathbf{b}$ is solved in two steps using forward and backward substitution:

\begin{enumerate}
    \item \textit{Forward Substitution} \\
    Solve the intermediate system $\mathbf{L}\mathbf{y} = \mathbf{b}$ for $\mathbf{y}$
    
    \item \textit{Backward Substitution} \\
    Solve the system $\mathbf{U}\mathbf{x} = \mathbf{y}$ for the solution vector $\mathbf{x}$.
\end{enumerate}
This two-step process efficiently leverages the triangular structure of $\mathbf{L}$ and $\mathbf{U}$, making the solution of the system computationally straightforward once the factorization is complete.

\subsubsection{QR Factorization}
QR factorization decomposes $\mathbf{A}$ into an orthogonal matrix $\mathbf{Q}$ and an upper triangular matrix $\mathbf{R}$ such that: 
\begin{equation}
\mathbf{A} = \mathbf{QR}
\end{equation}
Since $\mathbf{Q}$ is orthogonal, we can then transform the system to upper triangular form at the cost of a matrix-vector product:
\begin{equation}
\mathbf{R} = \mathbf{Q}^\top \mathbf{b}
\end{equation}
In general, QR factorization tends to be more expensive than LU factorization due to the added complexity of finding an orthogonal matrix $\mathbf{Q}$. However, this method is still particularly useful as it enables the solution of linear least squares problems for matrices that are not square.

\subsubsection{Symbolic and Numerical Factorizations}
The factorization process is divided into two key phases: symbolic and numerical factorization.
Symbolic factorization aims to identify the sparsity structure of the factor matrices (e.g., $\mathbf{L}$ and $\mathbf{U}$ in LU factorization) without considering the numerical values. It determines where non-zero elements will appear in the factors in order to minimize memory usage and computational cost. During this phase, a fill-in (the new non-zero entries that emerge during factorization) minimizing permutation of the matrix is computed using reordering algorithms like Minimum Degree or Nested Dissection. Numerical factorization uses the structure from the symbolic phase to generate the actual values of the factor matrices. Since the sparsity pattern is already known, the solver can efficiently perform the arithmetic operations required to obtain the factors. This phase is computationally intensive, and its efficiency depends on a multitude of factors, including the effectiveness of the symbolic factorization in minimizing fill-in. For systems with identical sparsity pattern, the symbolic phase only needs to be performed once, allowing repeated numerical factorizations at a lower computational cost.

\subsubsection{Row and Column Scaling}
Row and column scaling are preconditioning techniques often employed with direct solvers to improve the conditioning of a matrix before attempting a factorization. The goal is to reduce the condition number by scaling all rows/column to have unit norm, which can enhance the numerical stability of the factorization process and the accuracy of the solution.

\paragraph{Row Scaling}
Row scaling is achieved via left preconditioning by  setting $\mathbf{P} = \mathbf{D}_\mathrm{row}$, where $\mathbf{D}_\mathrm{row}$ is a diagonal matrix containing the norms of the rows of $\mathbf{A}$. The original system $\mathbf{A}\mathbf{x} = \mathbf{b}$ is then transformed into the preconditioned system:
\begin{equation}
(\mathbf{D}_\mathrm{row}^{-1} \mathbf{A}) \mathbf{x} = \mathbf{D}_\mathrm{row}^{-1} \mathbf{b}
\end{equation}
This transformation balances the influence of each row and is particularly beneficial when $\mathbf{A}$ has imbalanced row norms.

\paragraph{Column Scaling}
Column scaling is achieved via right preconditioning by setting $\mathbf{Q} = \mathbf{D}_\mathrm{col}$, where $\mathbf{D}_\mathrm{col}$ is a diagonal matrix containing the norms of the columns of $\mathbf{A}$. The original system is transformed as follows:
\begin{equation}
(\mathbf{A} \mathbf{D}_\mathrm{col}^{-1}) (\mathbf{D}_\mathrm{col} \mathbf{x}) = \mathbf{b}
\end{equation}
This can be solved in two steps:
\begin{enumerate}
    \item Solve the intermediate system $(\mathbf{A} \mathbf{D}_\mathrm{col}^{-1}) \mathbf{y} = \mathbf{b}$ for $\mathbf{y}$.
    \item Recover the solution $\mathbf{x}$ by computing $\mathbf{x} = \mathbf{D}_\mathrm{col}^{-1} \mathbf{y}$.
\end{enumerate}
This transformation balances the influence of each column and is particularly beneficial when $\mathbf{A}$ has imbalanced column norms.

\subsubsection{Complexity}
The computational complexity of direct solvers depends on the factorization method. For sparse matrices, the complexity is typically between $O(n^{1.5})$ and $O(n^2)$~\cite{sparse_QR_LU_complexity}, where $n$ is the dimension of the matrix. The exact complexity can vary based on the sparsity pattern and fill-in during the factorization process.

\subsubsection{Convergence and Stability}
Direct solvers are generally more robust than iterative solvers and provide exact solutions (subject to numerical precision). They are less sensitive to the conditioning of the matrix $\mathbf{A}$. However, they can suffer from numerical instability due to round-off errors, especially in ill-conditioned systems. Techniques like pivoting in LU factorization can help mitigate these issues.

\subsubsection{Strengths}
Direct solvers have notable advantages. They are robust, being less sensitive to the conditioning of the matrix and providing exact solutions up to numerical precision. Additionally, they are versatile, effective for a wide range of problems, including those with multiple right-hand sides or requiring factorizations to be reused.

\subsubsection{Weaknesses}
Direct solvers also have some disadvantages. The factorization process can be computationally expensive, particularly for large sparse matrices with significant fill-in. Furthermore, while parallel implementations of direct solvers exist, achieving efficient parallel performance is more challenging compared to iterative methods. Finally, they require higher memory usage since they need storage for the entire matrix and its factors, which can be prohibitive for very large systems.

\subsection{Overview of the Evaluated Libraries}
Table~\ref{tb:solver_overview} offers an overview of the numerical libraries and algorithms that were considered in this work. Both iterative and direct solvers were tested to provide a comprehensive evaluation of their performance and accuracy. We included libraries that support serial computation as well as those that utilize parallel computing through BLAS/LAPACK, OpenMP, and MPI. Our evaluation covered both CPU and GPU libraries to assess their capabilities on different hardware architectures. For iterative solvers, we focused on methods suitable for unsymmetric systems. In the case of direct solvers, we primarily targeted LU-based methods, as we were not interested in solving overdetermined linear systems. In total, 6 different numerical libraries were evaluated, for a total of 13 different methods.

\subsubsection{A Note on Iterative Solvers}
The evaluated iterative solvers were all implemented in the Eigen Library. The data indicates that for such ill-posed problems, no iterative method produced a correct solution. Although similar algorithms are available in other libraries, we did not perform any further testing due to these extremely poor accuracy results. Nevertheless, it should be noted that other libraries exist and may offer better performance or stability in different contexts. The most well-known are MAGMA~\cite{magma_1, magma_2}, STRUMPACK~\cite{strumpack}, PETSc~\cite{petsc-web-page}, Trilinos~\cite{trilinos-website}, and Hypre~\cite{hypre_library}.

\subsubsection{A Note on Distributed Solvers}
Given that the benchmark numerical experiment does not support MPI parallelism, the solvers were either compiled without MPI support or were run using a single rank with OpenMP parallelism instead. This approach allowed us to evaluate the solvers within the constraints of our experimental setup.

\subsubsection{A Note on cuSolver}
The cuSolver library included in the CUDA Toolkit was tested with QR factorization instead of LU factorization. Despite the GPU LU factorization routine being clearly documented~\cite{cuSolver}, it appears that it was never actually implemented. The declaration is missing from the cuSolver headers, and the function is not present in cuSolver 12.4. 

\include{tables/solver_table}

\section{Error estimates and conditioning}
This section discusses key concepts related to conditioning and error estimation when solving linear systems involving severely ill-conditioned matrices. We begin by defining conditioning and examining its impact on the accuracy of numerical solutions. We then derive relative error bounds for the solution of ill-conditioned linear systems: these bounds are essential in order to reliably evaluate the accuracy of numerical libraries and to assess the correctness of computed solutions. Finally, we define the relationship between matrix conditioning and numerical singularity, showing how poorly conditioned full-rank matrices can become singular when represented with limited precision float point numbers.

\subsection{Condition Number}

Conditioning is a measure of how the solution to a system of linear equations will change with respect to small changes in the input data. Mathematically, the condition number of a matrix \(A\) is defined as:
\begin{equation}
    \label{eq:kappa}
    \kappa_p(\mathbf{A}) = \|\mathbf{A}\|_p \cdot \|\mathbf{A}^{-1}\|_p
\end{equation}
where $\|A\|_p$ is a p-norm (usually $p = 1, 2, \infty$) of the matrix $A$ and $\|A^{-1}\|$ is the norm of its inverse. A high condition number indicates that the matrix is ill-conditioned, meaning that small changes in the input can lead to large changes in the solution $x$. Conditioning directly affects the stability and accuracy of numerical solutions to linear systems. It is important to note that conditioning is not related to numerical round-off errors introduced by computer floating point arithmetic, but rather is an intrinsic property of the matrix which represents the amplification of error.

\subsection{Computing Bounds on the Relative Error}
In this section we derive bounds on the relative error of the solution of a linear systems. When solving ill-conditoned sparse linear systems it is important to monitor these quantities in order to ensure the reliability of the obtained solution.

\begin{theorem} [Bounds on the relative error  $\|\mathbf{\hat x} - \mathbf{x}\| / \|\mathbf{x}\|$]
\label{th:error_x}
Given a linear system
\begin{equation}
\label{eq:lse}
\mathbf{Ax} = \mathbf{b}
\end{equation}
and an approximate solution
\begin{equation*}
\mathbf{\hat x} = \mathbf{x} + \Delta\mathbf{x}
\end{equation*}
such that:
\begin{gather*}
\mathbf{A\hat x} = \mathbf{\hat b} \\
\mathbf{\hat b} = \mathbf{b} + \Delta\mathbf{ b}
\end{gather*}
Then the relative error $\|\mathbf{\hat x} - \mathbf{x}\| / \|\mathbf{x}\|$ of the approximate solution $\mathbf{\hat x}$ to the true solution $\mathbf{x}$ is bounded by:
\begin{equation}
    \frac{1}{\kappa(\mathbf{A})} \frac{\|\mathbf{\hat b} - \mathbf{b} \|}{\|\mathbf{b}\|} \leq \frac{\|\mathbf{\hat x} - \mathbf{x}\|}{\|\mathbf{x}\|} \leq \kappa(\mathbf{A}) \frac{\|\mathbf{\hat b} - \mathbf{b} \|}{\|\mathbf{b}\|}
\end{equation}
where $\kappa(\mathbf{A})$ is the condition number of the matrix $\mathbf{A}$ as defined in eq.~\ref{eq:kappa}.

\end{theorem}
\begin{proof}
\textbf{Upper Bound}
\begin{align*}
\|\mathbf{\hat x} - \mathbf{x}\| &= \|\mathbf{A}^{-1}(\mathbf{A}\mathbf{\hat x} - \mathbf{A}\mathbf{x})\| = \\
&= \|\mathbf{A}^{-1}(\mathbf{\hat b} - \mathbf{b})\| \leq \|\mathbf{A}^{-1}\| \|\mathbf{\hat b} - \mathbf{b}\|
\end{align*}
Divide both sides of the inequality by $\|\mathbf{x}\|$:
\begin{align*}
\frac{\|\mathbf{\hat x} - \mathbf{x}\|}{\|\mathbf{x}\|} \leq \frac{\|\mathbf{A}^{-1}\| \|\mathbf{\hat b} - \mathbf{b}\|}{\|\mathbf{x}\|} = \kappa(\mathbf{A})\frac{\|\mathbf{\hat b} - \mathbf{b}\|}{\|\mathbf{A}\| \|\mathbf{x}\|}
\end{align*}
Notice that:
\begin{equation*}
\|\mathbf{b}\| = \|\mathbf{Ax}\| \leq \|\mathbf{A}\| \|\mathbf{x}\|
\end{equation*}
Thus:
\begin{align*}
\frac{\|\mathbf{\hat x} - \mathbf{x}\|}{\|\mathbf{x}\|} \leq \kappa(\mathbf{A}) \frac{\|\mathbf{\hat b} - \mathbf{b}\|}{\|\mathbf{b}\|}
\end{align*}

\noindent\textbf{Lower Bound}
\begin{align*}
\|\mathbf{\hat b} - \mathbf{b}\| &= \|\mathbf{A} (\mathbf{\hat x} - \mathbf{x})\| \leq \|\mathbf{A}\| \|\mathbf{\hat x} - \mathbf{x}\|
\end{align*}
Divide both sides of the inequality by $\|\mathbf{x}\|$ and rearrange:
\begin{align*}
\frac{\|\mathbf{\hat x} - \mathbf{x}\|}{\|\mathbf{x}\|} \geq \frac{\|\mathbf{\hat b} - \mathbf{b}\|}{\|\mathbf{A}\|\|\mathbf{x}\|} = \frac{\|\mathbf{\hat b} - \mathbf{b}\|}{\|\mathbf{A}\|\|\mathbf{A}^{-1}\mathbf{b}\|}
\end{align*}
Notice that:
\begin{equation*}
\|\mathbf{A}^{-1}\mathbf{b}\| \leq \|\mathbf{A}\|^{-1} \|\mathbf{b}\|
\end{equation*}
Thus:
\begin{align*}
\frac{\|\mathbf{\hat x} - \mathbf{x}\|}{\|\mathbf{x}\|} \geq  \frac{\|\mathbf{\hat b} - \mathbf{b}\|}{\|\mathbf{A}\|\|\mathbf{A}^{-1}\| \|\mathbf{b}\|} = \frac{1}{\kappa(\mathbf{A})}\frac{\|\mathbf{\hat b} - \mathbf{b}\|}{\|\mathbf{b}\|}
\end{align*}
$\qedsymbol$
\end{proof}

Theorem~\ref{th:error_x} is extremely important as it suggests that a small relative residual implies an accurate solution only if the conditioning number $\kappa(\mathbf{A})$ is small. This means that simply monitoring $\|\mathbf{A\hat x} - \mathbf{b}\| / \|\mathbf{b} \|$ as is often done does not guarantee correctness of the solution: even if this quantity approaches machine epsilon, the solution can still be completely meaningless.
The drawback of the bounds derived in Theorem~\ref{th:error_x} is that they are often not tight, leading to potentially correct solutions being flagged as unreliable. However, by considering the relative error $\|\mathbf{\hat x} - \mathbf{x}\| / \|\mathbf{\hat x}\|$ with respect to $\|\mathbf{\hat x}\|$ instead, it is possible to derive tighter bounds.

\begin{theorem} [Bounds on the relative error  
$\|\mathbf{\hat x} - \mathbf{x}\| / \| \mathbf{\hat x}\|$]
\label{th:err_hatx}
Given an approximate solution $\mathbf{\hat x}$ to a linear system as described in Theorem~\ref{th:error_x}, the relative error $\|\mathbf{\hat x} - \mathbf{x}\| / \| \mathbf{\hat x}\|$ is bounded by:
\begin{equation}
\frac{\|\mathbf{\hat b} - \mathbf{b}\|}{\|\mathbf{A}\| \|\mathbf{\hat x}\|} \leq \frac{\|\mathbf{\hat x} - \mathbf{x}\|}{\|\mathbf{\hat x}\|} \leq  \kappa(\mathbf{A})\frac{\|\mathbf{\hat b} - \mathbf{b}\|}{\|\mathbf{A}\| \|\mathbf{\hat x}\|}
\end{equation}
\end{theorem}

\begin{proof}
\textbf{Upper Bound} 

\noindent We have shown in Theorem~\ref{th:error_x} that:
\begin{align*}
\|\mathbf{\hat x} - \mathbf{x}\|  \leq \|\mathbf{A}^{-1}\| \|\mathbf{\hat b} - \mathbf{b}\|
\end{align*}
Divide both sides of the inequality by $\|\mathbf{\hat x}\|$:
\begin{align*}
\frac{\|\mathbf{\hat x} - \mathbf{x}\|}{\|\mathbf{\hat x}\|} \leq \frac{\|\mathbf{A}^{-1}\| \|\mathbf{\hat b} - \mathbf{b}\|}{\|\mathbf{\hat x}\|} = \kappa(\mathbf{A})\frac{\|\mathbf{\hat b} - \mathbf{b}\|}{\|\mathbf{A}\| \|\mathbf{\hat x}\|}
\end{align*}

\noindent\textbf{Lower Bound}
\noindent Similarly, we have shown in Theorem~\ref{th:error_x} that:
\begin{align*}
\|\mathbf{\hat b} - \mathbf{b}\| \leq \|\mathbf{A}\| \|\mathbf{\hat x} - \mathbf{x}\|
\end{align*}
Rearrange and divide both sides of the inequality by $\|\mathbf{\hat x}\|$:
\begin{align*}
\frac{\|\mathbf{\hat x} - \mathbf{x}\|}{\|\mathbf{\hat x}\|} \geq \frac{\|\mathbf{\hat b} - \mathbf{b}\|}{\|\mathbf{A}\|\|\mathbf{\hat x}\|}
\end{align*}
\qedsymbol
\end{proof}

Theorem~\ref{th:err_hatx} gives a tight bound on the error relative to $\|\mathbf{\hat x}\|$. However, since we are interested in the deviation from $\mathbf{x}$, it is important to note what normalizing with respect to $\mathbf{\hat x}$ implies. Theorem~\ref{th:err_x_tight} proves that a tight upper bound on the  relative error $\|\mathbf{\hat x} - \mathbf{x}\| / \| \mathbf{\hat x}\|$  guarantees a tight bound on the error $\|\mathbf{\hat x} - \mathbf{x}\| / \| \mathbf{x}\|$.

\begin{theorem} [Tight upper bound of the error   
$\|\mathbf{\hat x} - \mathbf{x}\| / \| \mathbf{x}\|$]
\label{th:err_x_tight}
Let $\varepsilon = \|\mathbf{\hat x} - \mathbf{x}\| / \| \mathbf{\hat x}\|$ be the error of an approximate solution to a linear system relative to the approximate solution $\mathbf{\hat x}$. Then the error $\|\mathbf{\hat x} - \mathbf{x}\| / \| \mathbf{x}\|$ relative to the true solution $\mathbf{x}$ is bounded by:
\begin{equation}
\label{eq:epsilon_bound}
\frac{\|\mathbf{\hat x} - \mathbf{x}\|}{ \| \mathbf{x}\|} \leq \frac{\varepsilon}{1 - \varepsilon}
\end{equation}
\end{theorem}

\begin{proof}
\begin{gather*}
    \|\mathbf{\hat x}\| = \|\mathbf{\hat x} - \mathbf{x} + \mathbf{x}\| \leq\|\mathbf{\hat x} - \mathbf{x}\| + \|\mathbf{x}\|
\end{gather*}
Divide both sides by $\|\mathbf{\hat x}\|$:
\begin{gather*}
1 \leq \frac{\|\mathbf{\hat x} - \mathbf{x}\|}{\|\mathbf{\hat x}\|} + \frac{\|\mathbf{x}\|}{\|\mathbf{\hat x}\|} \\
1 \leq \varepsilon + \frac{\|\mathbf{x}\|}{\|\mathbf{\hat x}\|} \Leftrightarrow
 1 - \varepsilon \leq \frac{\|\mathbf{x}\|}{\|\mathbf{\hat x}\|}
\end{gather*}
Assume $\varepsilon < 1$:
\begin{gather*}
  \frac{\|\mathbf{\hat x}\|}{\|\mathbf{x}\|} \leq  \frac{1}{1 - \varepsilon}
\end{gather*}
Notice that:
\begin{gather*}
\frac{\|\mathbf{\hat x}\|}{\|\mathbf{x}\|} = \frac{1}{\varepsilon}\frac{\|\mathbf{\hat x} - \mathbf{x}\|}{\|\mathbf{x}\|}
\end{gather*}
Substitute and reorder:
\begin{gather*}
\frac{1}{\varepsilon}\frac{\|\mathbf{\hat x} - \mathbf{x}\|}{\|\mathbf{x}\|} \leq \frac{1}{1 - \varepsilon}\\
\frac{\|\mathbf{\hat x} - \mathbf{x}\|}{\|\mathbf{x}\|} \leq \frac{\varepsilon}{1 - \varepsilon}
\end{gather*}
\qedsymbol
\end{proof}
For $\varepsilon \ll 1$, eq.~\ref{eq:epsilon_bound} converges approximately to $\varepsilon$:
\begin{equation*}
\frac{\|\mathbf{\hat x} - \mathbf{x}\|}{ \| \mathbf{x}\|} \leq \frac{\varepsilon}{1 - \varepsilon} \approx \varepsilon
\end{equation*}
This proves that, so long as the tight upper bound on \mbox{$\|\mathbf{\hat x} - \mathbf{x}\| / \| \mathbf{\hat x}\|$} is much smaller than $1$, then the error relative to the true solution will also be similarly small. However, as soon as $\varepsilon$ grows towards $1$, then the error relative to $\mathbf{x}$ quickly becomes unbounded.
The key takeaways of this error analysis are thus the following:
\begin{itemize}
    \item {The condition number $\kappa(\mathbf{A})$ quantifies the amplification of the backward error from the residual vector to the true solution and is a property intrinsic to the matrix and unrelated to finite-precision arithmetic.}
    \item {Checking the correctenss of a solution by computing the relative residual $\|\mathbf{A\hat x} - \mathbf{b}\| / \|\mathbf{b} \|$ as is often done is misleading and not a reliable approach.}
    \item {A reliable way of estimating the relative error of the true solution is to compute the upper bound of the error relative to $\mathbf{\hat x}$: $\frac{\|\mathbf{\hat x} - \mathbf{x}\|}{\|\mathbf{\hat x}\|} \leq \varepsilon =   \kappa(\mathbf{A})\frac{\|\mathbf{\hat b} - \mathbf{b}\|}{\|\mathbf{A}\| \|\mathbf{\hat x}\|}$. So long as $\varepsilon \ll 1$ is small enough (e.g. $\varepsilon \leq 10^{-7}$)}, then it is guaranteed that the error relative to the true solution is similarly small.
\end{itemize}
By following these three points, it is possible to ensure that the accuracy of an approximate solution is evaluated in a robust and reliable way.

\subsection{Condition Number and Numerical Singularity}

In practice, solving a linear system is done using floating-point arithmetic, which inevitably introduces rounding errors due to the finite precision of the data type. This section explores the relationship between round-off error, matrix conditioning, and numerical singularity.

Consider the scenario where we solve a linear system as described in Theorem~\ref{th:error_x} using floating-point arithmetic. Suppose we have an idealized algorithm that operates with infinite precision throughout the solution process and only applies rounding at the final step. Theorem~\ref{th:numerical_singularity} establishes a threshold for the conditioning of $\mathbf{A}$, beyond which the matrix is deemed numerically singular, rendering the computed solution $\mathbf{\hat{x}}$ unreliable.

\begin{theorem}[Numerical singularity of a matrix]
\label{th:numerical_singularity}
\theoremnewline
Given a linear system $\mathbf{Ax}=\mathbf{b}$ and its finite-precision representation
\begin{equation*}
\mathbf{\hat A}\mathbf{\hat x}=\mathbf{\hat b}
\end{equation*}
such that
\begin{gather*}
\mathbf{\hat A} =\mathbf{A}+\Delta \mathbf{A}, \ \ \mathbf{\hat x} =\mathbf{x}+\Delta \mathbf{x}, \ \ \mathbf{\hat b} =\mathbf{b}+\Delta \mathbf{b} \\
\frac{\|\mathbf{\hat A} - \mathbf{A} \|}{\|\mathbf{A}\|} \leq \varepsilon, \quad \frac{\|\mathbf{\hat b} - \mathbf{b} \|}{\|\mathbf{b}\|} \leq \varepsilon
\end{gather*}
where $\varepsilon$ is the machine precision of a specific floating point datatype, then the irreducible error introduced by rounding operations is bounded by:

\begin{equation}
\label{eq:numerical_singularity}
\frac{\|\mathbf{\hat x} - \mathbf{x}\|}{\|\mathbf{x}\|} \leq 2\varepsilon\frac{\kappa(\mathbf{A})}{1 - \kappa(\mathbf{A})\varepsilon}
\end{equation}
\end{theorem}
\begin{proof}
\begin{gather*}
\mathbf{\hat A}\mathbf{\hat x}=\mathbf{\hat b} \\
(\mathbf{A} + \Delta \mathbf{A}) (\mathbf{x} + \Delta \mathbf{x}) = \mathbf{b} + \Delta \mathbf{b} \\
\cancel{\mathbf{Ax}} + \Delta\mathbf{Ax} + \mathbf{A}\Delta\mathbf{x} + \Delta \mathbf{A}\Delta\mathbf{x}= \cancel{\mathbf{b}} + \Delta\mathbf{b} \\
\mathbf{A}\Delta\mathbf{x} = \Delta\mathbf{b} - \Delta\mathbf{Ax} -  \Delta \mathbf{A}\Delta\mathbf{x} \\
\Delta\mathbf{x} = \mathbf{A}^{-1}(\Delta\mathbf{b} - \Delta\mathbf{Ax} -  \Delta \mathbf{A}\Delta\mathbf{x})\\
\end{gather*}
Take the norm of both sides
\begin{align*}
\|\Delta\mathbf{x}\| &= \|\mathbf{A}^{-1}(\Delta\mathbf{b} - \Delta\mathbf{Ax} -  \Delta \mathbf{A}\Delta\mathbf{x})\| \\
&\leq \|\mathbf{A}^{-1}\| \|\Delta\mathbf{b} - \Delta\mathbf{Ax} -  \Delta \mathbf{A}\Delta\mathbf{x}\| \\
& \leq\|\mathbf{A}^{-1}\| (\|\Delta\mathbf{b}\| + \|\Delta\mathbf{A}\| \|\mathbf{x}\| +  \|\Delta \mathbf{A}\| \|\Delta\mathbf{x}\|)
\end{align*}
Move $\Delta \mathbf{x}$ terms to the left
\begin{gather*}
(1 - \|\mathbf{A}^{-1}\| \|\Delta \mathbf{A}\|) \|\Delta\mathbf{x}\|  \leq \|\mathbf{A}^{-1}\| (\|\Delta\mathbf{b}\| + \|\Delta\mathbf{A}\| \|\mathbf{x}\|)
\end{gather*}
Assuming $\|\mathbf{A}^{-1}\| \|\Delta \mathbf{A}\| < 1$, rearrange and divide both sides by $\| \mathbf{x}\|$
\begin{gather*}
\frac{\|\Delta\mathbf{x}\|}{\|\mathbf{x}\|}  \leq \frac{\|\mathbf{A}^{-1}\|}{1 - \|\mathbf{A}^{-1}\| \|\Delta \mathbf{A}\|} \left(\frac{\|\Delta\mathbf{b}\|}{\|\mathbf{x}\|} + \|\Delta\mathbf{A}\|\right)
\end{gather*}
Introduce $\|\mathbf{A}\|$ twice on the right
\begin{equation*}
\frac{\|\Delta\mathbf{x}\|}{\|\mathbf{x}\|}  \leq \frac{\|\mathbf{A}^{-1}\|\|\mathbf{A}\|}{1 - \|\mathbf{A}^{-1}\|\|\mathbf{A}\| \frac{\|\Delta \mathbf{A}\|}{\|\mathbf{A}\|}} \left(\frac{\|\Delta\mathbf{b}\|}{\|\mathbf{A}\|\|\mathbf{x}\|} + \frac{\|\Delta\mathbf{A}\|}{\|\mathbf{A}\|}\right)
\end{equation*}
Notice that:
\begin{equation*}
\|\mathbf{b}\| = \|\mathbf{Ax}\| \leq \|\mathbf{A}\| \|\mathbf{x}\|
\end{equation*}
Thus, after a few substitutions, we obtain:
\begin{equation*}
\frac{\|\mathbf{\hat x} - \mathbf{x}\|}{\|\mathbf{x}\|}  \leq \frac{\kappa(\mathbf{A})}{1 - \kappa(\mathbf{A}) \frac{\|\Delta \mathbf{A}\|}{\|\mathbf{A}\|}} \left(\frac{\|\Delta\mathbf{b}\|}{\|\mathbf{b}\|} + \frac{\|\Delta\mathbf{A}\|}{\|\mathbf{A}\|}\right)
\end{equation*}
Finally, introducing the machine epsilon $\varepsilon$, we reach
\begin{equation*}
\frac{\|\mathbf{\hat x} - \mathbf{x}\|}{\|\mathbf{x}\|}  \leq \frac{\kappa(\mathbf{A})}{1 - \kappa(\mathbf{A})\varepsilon} (\varepsilon + \varepsilon) = 2\varepsilon\frac{\kappa(\mathbf{A})}{1 - \kappa(\mathbf{A})\varepsilon} 
\end{equation*}
\qedsymbol
\end{proof}
Clearly as the condition number of $\kappa(\mathbf{A})$ approaches $1/\varepsilon$ eq.~\ref{eq:numerical_singularity} diverges towards $+\infty$. Practically this means that the computed solution $\mathbf{\hat x}$ becomes increasingly unreliable as the effects of round-off errors become more prominent. When $\kappa(\mathbf{A})$ becomes equal to $1/\varepsilon$, the relative error of the solution becomes unbounded and the matrix $\mathbf{A}$ must be considered numerically singular as it can potentially be rounded to a singular $\mathbf{\hat A}$.

\section{Computing the condition number of a matrix}

Accurately estimating the $\mathcal{L}_2$ condition number of a matrix is crucial for bounding the relative $\mathcal{L}_2$ error of the solution. Efficient methods like Hager's algorithm~\cite{hager_conditioning} are used to estimate the $\mathcal{L}_1$ or $\mathcal{L}_\infty$ condition numbers in libraries such as LAPACK~\cite{lapack1987} and Intel oneAPI MKL~\cite{intel2024gecon}. The algorithm works by iteratively searching a solution to the maximization problem $\max \| \mathbf{A}^{-1} \mathbf{v}\|_1$ subject to the constraint $\|\mathbf{v}\|_1=1$, which is convex but not differentiable. These estimates can vary by orders of magnitude from the true value, which may suffice for many applications but not for accuracy benchmarks. Moreover, to the best of our knowledge, no implementation of this method is currently publicly available for sparse matrices. NumPy and MATLAB compute the $\mathcal{L}_2$ condition number accurately using an alternative definition based on the singular values of a matrix. In fact, it can be shown that: 
\begin{equation}
\kappa_2(\mathbf{A}) = \|\mathbf{A}\|_2 \cdot \|\mathbf{A}^{-1}\|_2 = \frac{\max(\sigma(\mathbf{A}))}{\min(\sigma(\mathbf{A}))}
\end{equation}
where $\sigma(\mathbf{A})$ are the singular values of $\mathbf{A}$~\cite{matrix_cookbook}. In practice, these libraries compute $\kappa_2(\mathbf{A})$ by means of a dense SVD decomposition~\cite{numpy_cond, matlab_cond}. While accurate, this approach is computationally expensive and not scalable in size. For large sparse matrices, methods to estimate $\kappa_2(\mathbf{A})$ usually rely on iterative eigenvalue algorithms like the power method and inverse iteration~\cite{saad2011single_vector_iteration} or the Lanczos algorithm~\cite{saad2011hermitian_lanczos} applied to an implicitly constructed $\mathbf{A}^\top \mathbf{A}$ matrix. The eigenvalues of $\mathbf{A}^\top \mathbf{A}$ correspond to the squares of the singular values of $\mathbf{A}$. Once the extremal eigenvalues of $\mathbf{A}^\top \mathbf{A}$ are known, singular values are computed as $\sqrt{\lambda_{\max}}$ and $\sqrt{\lambda_{\min}}$, where $\lambda_{\max}$ and $\lambda_{\min}$ are the largest and smallest eigenvalues in magnitude. However, these methods struggle with convergence for very ill-conditioned matrices, which is precisely the case critical for our benchmarks.

In this section we propose a novel method leveraging an iterative projected gradient descent algorithm, accelerated by momentum and inspired by the Adam optimizer~\cite{adam}, very commonly used in machine learning applications. The method is used to accurately compute the $\mathcal{L}_2$ norms of $\mathbf{A}$ and $\mathbf{A}^{-1}$, which are then multiplied to form $\kappa_2(\mathbf{A})$. Similarly to Hager's algorithm, in order to compute $\|\mathbf{A}^{-1}\|$ the method requires a single sparse LU factorization of $\mathbf{A}$. However, In stark contrast to Hager's method, our proposed algorithm aims to solve differentiable minimisation and maximisation problems by relaxing the constraints on the norm of the solution. Convergence is then accelerated by projecting the sequence of approximate solutions in order to implicitly enforce the constraints.
Our preliminary tests indicate that the method is suitable for large-scale and severely ill-conditioned sparse systems. Moreover, contrarily to the power and inverse iteration methods, our technique was able to accurately compute condition numbers in the order of $10^{26}$. In this section, we present our algorithm and apply it to the benchmark problem considered in this work. Further testing and numerical experiments will be the subject of an upcoming paper.

\subsection{Problem Formulation}

\noindent While traditional algorithms approximate $\kappa_2(\mathbf{A})$ by reformulating the problem in terms of computing the extremal singular values of $\mathbf{A}$, we propose to tackle an alternative definition. The $\mathcal{L}_2$ operator-induced norm is defined as~\cite{matrix_cookbook}:
\begin{equation}
\label{eq:op_norm_def}
\|\mathbf{A}\|_2 = \sup_{\|\mathbf{v}\|_2 = 1} \|\mathbf{A}\mathbf{v}\|_2
\end{equation}
This definition is not very useful right-away, however it is possible to recast equation~\ref{eq:op_norm_def} to a more tractable form.

\begin{theorem}[Equivalent definitions of the $\mathcal{L}_2$ matrix norm]
\label{th:op_norm}
Let $\mathbf{A}$ be a matrix in $\mathbb{R}^{n \times n}$, and let $\|\cdot\|_2$ denote the  $\mathcal{L}_2$ norm (Euclidean norm) in $\mathbb{R}^n$. The operator-induced $\mathcal{L}_2$ norm is given by both
\begin{equation}
\label{eq:2_norm_def_1}
\|\mathbf{A}\|_2 = \sup_{\|\mathbf{v}\|_2 = 1} \|\mathbf{A}\mathbf{v}\|_2
\end{equation}
and
\begin{equation}
\label{eq:2_norm_def_2}
\|\mathbf{A}\|_2 = \sup_{\mathbf{v} \neq \mathbf{0}} \frac{\|\mathbf{A}\mathbf{v}\|_2}{\|\mathbf{v}\|_2}
\end{equation}
These two definitions are equivalent.
\end{theorem}

\begin{proof}
We will prove that the two definitions are equivalent by showing that:
\begin{equation}
\label{eq:sup_equality}
\sup_{\|\mathbf{v}\|_2 = 1} \|\mathbf{A}\mathbf{v}\|_2 = \sup_{\mathbf{v} \neq \mathbf{0}} \frac{\|\mathbf{A}\mathbf{v}\|_2}{\|\mathbf{v}\|_2}
\end{equation}
\textbf{Show that} $\sup_{\|\mathbf{v}\|_2 = 1} \|\mathbf{A}\mathbf{v}\|_2 \leq \sup_{\mathbf{v} \neq \mathbf{0}} \frac{\|\mathbf{A}\mathbf{v}\|_2}{\|\mathbf{v}\|_2}$:

\vspace{0.1cm}
\noindent For any vector $\mathbf{v}$ such that $\|\mathbf{v}\|_2 = 1$, we have
\begin{equation*}
\frac{\|\mathbf{A}\mathbf{v}\|_2}{\|\mathbf{v}\|_2} = \|\mathbf{A}\mathbf{v}\|_2,
\end{equation*}
Therefore:
\begin{equation*}
\|\mathbf{A}\mathbf{v}\|_2 \leq \sup_{\mathbf{v} \neq \mathbf{0}} \frac{\|\mathbf{A}\mathbf{v}\|_2}{\|\mathbf{v}\|_2} \quad \text{for all } \mathbf{v} \text{ with } \|\mathbf{v}\|_2 = 1
\end{equation*}
Taking the supremum over all $\mathbf{v}$ with $\|\mathbf{v}\|_2 = 1$, we get
\begin{equation}
\label{eq:sup_inequality_1}
\sup_{\|\mathbf{v}\|_2 = 1} \|\mathbf{A}\mathbf{v}\|_2 \leq \sup_{\mathbf{v} \neq \mathbf{0}} \frac{\|\mathbf{A}\mathbf{v}\|_2}{\|\mathbf{v}\|_2}
\end{equation}
\textbf{Show that} $\sup_{\mathbf{v} \neq \mathbf{0}} \frac{\|\mathbf{A}\mathbf{v}\|_2}{\|\mathbf{v}\|_2} \leq \sup_{\|\mathbf{v}\|_2 = 1} \|\mathbf{A}\mathbf{v}\|_2$:

\vspace{0.1cm}
\noindent Consider any nonzero vector $\mathbf{v} \in \mathbb{R}^n$. Define $\mathbf{u} = \frac{\mathbf{v}}{\|\mathbf{v}\|_2}$, so that $\|\mathbf{u}\|_2 = 1$. Then:
\begin{equation}
\label{eq:sup_v_equivalence}
\frac{\|\mathbf{A}\mathbf{v}\|_2}{\|\mathbf{v}\|_2} = \|\mathbf{A}\left(\frac{\mathbf{v}}{\|\mathbf{v}\|_2}\right)\|_2 = \|\mathbf{A}\mathbf{u}\|_2
\end{equation}
Thus:
\begin{equation*}
\frac{\|\mathbf{A}\mathbf{v}\|_2}{\|\mathbf{v}\|_2} = \|\mathbf{A}\mathbf{u}\|_2,
\end{equation*}
where $\|\mathbf{u}\|_2 = 1$. Therefore:
\begin{equation*}
\frac{\|\mathbf{A}\mathbf{v}\|_2}{\|\mathbf{v}\|_2} \leq \sup_{\|\mathbf{u}\|_2=1} \|\mathbf{A}\mathbf{u}\|_2
\end{equation*}
Taking the supremum over all $\mathbf{v} \neq \mathbf{0}$ and renaming $\mathbf{u}$ to $\mathbf{v}$, we get:
\begin{equation}
\label{eq:sup_inequality_2}
\sup_{\mathbf{v} \neq \mathbf{0}} \frac{\|\mathbf{A}\mathbf{v}\|_2}{\|\mathbf{v}\|_2} \leq \sup_{\|\mathbf{v}\|_2=1} \|\mathbf{A}\mathbf{v}\|_2
\end{equation}
Since we have proven that both inequalities~\ref{eq:sup_inequality_1} and~\ref{eq:sup_inequality_2} hold, then it must also be that the equation~\ref{eq:sup_equality} holds as well. \qedsymbol
\end{proof}
Theorem~\ref{th:op_norm} gives a definition of $\|\mathbf{A}\|_2$ that is not explicitly subject to the constraint $\|\mathbf{v}\|_2 = 1$. Instead, this constraint is implicitly encoded in the denominator of the expression $\frac{\|\mathbf{A} \mathbf{v}\|_2}{\|\mathbf{v}\|_2}$. Since $\|\mathbf{A}\mathbf{v}\|_2$ is a continuous function and the set of unit vectors (the surface of the unit ball, \mbox{$B^n = \{\mathbf{v} \in \mathbb{R}^n \mid \|\mathbf{v}\|_2 = 1\}$}) is compact, the Extreme Value Theorem guarantees that $\|\mathbf{A}\mathbf{v}\|_2$ will attain its maximum on the surface of the unit ball. Moreover, Equation~\ref{eq:sup_v_equivalence} shows that $\frac{\|\mathbf{A}\mathbf{v}\|_2}{\|\mathbf{v}\|_2}$ is homogeneous, meaning that the vector $\mathbf{v}$ which maximizes $\|\mathbf{A}\mathbf{v}\|_2$ under the constraint $\|\mathbf{v}\|_2 = 1$ also attains the supremum for $\frac{\|\mathbf{A}\mathbf{v}\|_2}{\|\mathbf{v}\|_2}$ over all nonzero vectors $\mathbf{v}$. This allows us to formulate the following equivalence:
\begin{align}
\label{eq:max_sup_equivalence}
\begin{split}
    \|\mathbf{A}\|_2 &= \sup_{\|\mathbf{v}\|_2=1}\|\mathbf{A} \mathbf{v}\|_2 = \max_{\|\mathbf{v}\|_2=1}\|\mathbf{A} \mathbf{v}\|_2 =\\
    &= \sup_{\mathbf{v} \neq \mathbf{0}} \frac{\|\mathbf{A} \mathbf{v}\|_2}{\|\mathbf{v}\|_2} = \max_{\mathbf{v} \neq \mathbf{0}} \frac{\|\mathbf{A} \mathbf{v}\|_2}{\|\mathbf{v}\|_2}
\end{split}
\end{align}
Thus, the matrix norm $\|\mathbf{A}\|_2$ is equal to the maximum $\max_{\mathbf{v} \neq \mathbf{0}} \frac{\|\mathbf{A} \mathbf{v}\|_2}{\|\mathbf{v}\|_2}$, and this maximum is attained at the same vector that maximizes $\|\mathbf{A} \mathbf{v}\|_2$ on the unit sphere.
\subsection{Loss Function and Optimization}
Equation~\ref{eq:max_sup_equivalence} allows us to recast the computation of $\|\mathbf{A}\|_2$ as a maximization (minimization) problem of a continuous function. In the field of machine learning, these problems are solved on regular basis and over the years powerful and versatile optimizers have been developed. We propose to modify the Adam optimizer in order to inexpensively compute $\|\mathbf{A}\|_2$. To simplify the process, we aim to compute the vector $\mathbf{v}^*$ for which $\|\mathbf{A}\|_2 = \|\mathbf{Av}^*\|_2$. This allows us to leverage the monotonicity of the $\log$ function, further simplifying the problem:
\begin{align}
\begin{split}
    \mathbf{v}^* &= \arg \max_{\mathbf{v} \neq \mathbf{0}} \frac{\|\mathbf{A} \mathbf{v}\|_2}{\|\mathbf{v}\|_2} = \arg \max_{\mathbf{v} \neq \mathbf{0}} \log \left( \frac{\|\mathbf{A} \mathbf{v}\|_2}{\|\mathbf{v}\|_2}\right) \\ 
    &= \arg \min_{\mathbf{v} \neq \mathbf{0}} \log \left(\|\mathbf{v}\|_2\right) - \log\left( \|\mathbf{A} \mathbf{v}\|_2\right) 
\end{split}
\end{align}
We define the loss function $L$ as:
\begin{equation}
L(\mathbf{v}) = \log \left(\|\mathbf{v}\|_2\right) - \log\left( \|\mathbf{A} \mathbf{v}\|_2\right)     
\end{equation}
Notice that the gradient $\nabla L$ is readily available:
\begin{align}
\label{eq:nabla_loss}
\nabla L(\mathbf{v}) =  \frac{\mathbf{v}}{\|\mathbf{v}\|_2^2} -   \frac{\mathbf{A}^\top\mathbf{Av}}{\|\mathbf{Av}\|_2^2}
\end{align}
\begin{algorithm}[h!]
\SetAlgoLined
\KwIn{$L$, $\nabla L$, $A$, $\mathbf{x}_{\text{init}}$, $\alpha$, $\beta_1$, $\beta_2$, $\epsilon$, $\operatorname{max\_iter}$}
\KwOut{$\mathbf{x}_{\operatorname{min}}$ approximate minimizer of $L$}

$\mathbf{x} \gets \mathbf{x}_{\text{init}}$\;
$\mathbf{m} \gets \mathbf{0}$\;
$\mathbf{v} \gets \mathbf{0}$\;
$L_{\operatorname{min}} \gets \infty$\;
$\mathbf{x}_{\operatorname{min}} \gets \mathbf{x}$\;\

\For{$t \gets 1$ \textbf{to} $\operatorname{max\_iter}$}{

    \tcp{Compute gradient of the loss}
    $\mathbf{g} \gets \nabla L(\mathbf{x})$\;\
    
    \tcp{Update biased moments}
    $\mathbf{m} \gets \beta_1 \cdot \mathbf{m} + (1 - \beta_1) \cdot \mathbf{g}$\;
    $\mathbf{v} \gets \beta_2 \cdot \mathbf{v} + (1 - \beta_2) \cdot \mathbf{g}^2$\;\
    
    \tcp{Compute corrected moments}
    $\hat{\mathbf{m}} \gets \mathbf{m} / (1 - \beta_1^t)$\;
    $\hat{\mathbf{v}} \gets \mathbf{v} / (1 - \beta_2^t)$\;\
    
    \tcp{Compute update to the solution}
    $\mathbf{\Delta x} \gets \alpha \cdot \hat{\mathbf{m}} / (\sqrt{\hat{\mathbf{v}}} + \epsilon)$\;\
    
    \tcp{Project update to be tangent to the unit sphere}
    $\mathbf{\Delta x} \gets \mathbf{\Delta x} - (\mathbf{x}^\top \mathbf{\Delta x}) \cdot \mathbf{x}$\;\
    
    \tcp{Update solution}
    $\mathbf{x} \gets \mathbf{x} - \mathbf{\Delta x}$\;\
    
    \tcp{Normalize x to project back to the unit sphere}
    $\mathbf{x} \gets \mathbf{x} / \|\mathbf{x}\|$\;\
    
    \tcp{Optionally track loss}
    \If{$L(\mathbf{x}) <L_{\operatorname{min}}$}{
        $L_{\operatorname{min}} \gets L(\mathbf{x})$\;
        $\mathbf{x}_{\operatorname{min}} \gets \mathbf{x}$\;
    }
}

\Return $\mathbf{x}_{\operatorname{min}}$\;
\caption{Projected Adam Optimizer}
\label{alg:p_adam}
\end{algorithm}

In principle, it is possible to employ vanilla Adam to minimize $L$, yet in our experience this leads to slow convergence. Instead, we propose to leverage the information about the location of the maximiser (i.e. constraint $\|\mathbf{v}\|_2 = 1$ set by the definition of equation~\ref{eq:2_norm_def_1}) in order to decrease the number of iterations required. This can be done by implementing two small operations, yielding the projected Adam algorithm:
\begin{enumerate}
    \item{After the update $\mathbf{\Delta x}$ is computed, it is orthogonalized against the current approximation $\mathbf{x}$ of $\mathbf{v}^*$ to ensure that it is tangent to the surface of the unit sphere.}
    \item{After $\mathbf{x}$ is updated, it is projected back onto the surface of the unit sphere in order to ensure that it lies in the lower-dimensional domain where the solution is guaranteed to be. This is done by normalizing the vector to unit length.}
\end{enumerate}
The full Projected Adam algorithm is detailed in snippet~\ref{alg:p_adam}

\FloatBarrier
Computing the condition number $\kappa_2(\mathbf{A})$ requires both $\|\mathbf{A}\|_2$ as well as $\|\mathbf{A}^{-1}\|_2$. The latter can be computed via the Projected Adam method, however calculating the gradient $\nabla L$ explicitly via the definition given in equation~\ref{eq:nabla_loss} becomes prohibitively expensive. Luckily it is possible to compute the gradient implicitly by leveraging a single LU factorization $\mathbf{A} = \mathbf{LU}$. Once the $\mathbf{L}$ and $\mathbf{U}$ factors are known, it is possible to use them in order to compute $\nabla L$ as detailed in snippet~\ref{alg:computing_nabla_L}.
\begin{algorithm}[ht]
\SetAlgoLined
\KwIn{$\mathbf{x}$, $\mathbf{L}$, $\mathbf{U}$}
\KwOut{$\nabla L(\mathbf{x})$}

\tcp{Compute $\mathbf{r} = \mathbf{A}^{-1} \mathbf{x}$}
Solve $\mathbf{L}\mathbf{U}\mathbf{r} = \mathbf{x}$\;
\ \ \

\tcp{Compute $\|\mathbf{r}\| = \|\mathbf{A}^{-1} \mathbf{x}\|$}
$\operatorname{norm\_r} \gets \|\mathbf{r}\|$\;
\ \ \

\tcp{Compute $\mathbf{s} = \mathbf{A}^{-T}\mathbf{A}^{-1} \mathbf{x}$}
Solve $\mathbf{U}^\top \mathbf{L}^\top \mathbf{s} = \mathbf{r}$\;
\ \ \

\tcp{Compute $\|\mathbf{x}\|$}
$\operatorname{norm\_x} \gets \|\mathbf{x}\|$\;
\ \ \

\tcp{Compute the gradient $\nabla L(\mathbf{x})$}
$\nabla L(\mathbf{x}) \gets \frac{\mathbf{x}}{\operatorname{norm\_x}^2} - \frac{\mathbf{s}}{\operatorname{norm\_r}^2}$\;

\vspace{0.2cm} 
\Return $\nabla L(\mathbf{x})$\;
\caption{Compute Gradient $\nabla L(\mathbf{x})$ for $\mathbf{A}^{-1}$}
\label{alg:computing_nabla_L}
\end{algorithm}

\section{Benchmark Numerical Model}
In this section, we present the numerical model used as the benchmark to evaluate the performance of the different sparse linear solvers.
We begin by defining the governing equations of the system, detailing the PDEs that describe the physical phenomena being modeled. These equations provide the mathematical framework that encodes the problem. Next, we discuss the discretization strategy from which the benchmark numerical data is generated. Finally, we examine the structure and conditioning of the resulting linear systems, documenting the preconditioning strategy used in order to compute accurate solutions.
\subsection{Governing Equations}
\label{ch:governing_equations}
The system is derived from the numerical simulation recently proposed by Gerya in~\cite{Gerya2024}. The following conservation equations are solved in a 2D numerical domain representing a rectangular section of the lithosphere, sized $500 \operatorname{m} \times 50\operatorname{m}$:

\noindent \textbf{Conservation of mass for solid:}
\begin{equation}
\begin{aligned} \operatorname{div}\left(\vec{v}_{\text {solid }}\right) & +\beta_d\left(\frac{D^s P_{\text {total }}}{D t}-\alpha \frac{D^f P_{\text {fluid }}}{D t}\right) \\ & +\frac{P_{\text {total }}-P_{\text {fluid }}}{(1-\phi) \eta_\phi}=0\end{aligned}
\end{equation}

\noindent \textbf{Conservation of mass for fluid:}
\begin{equation}
\begin{aligned}\operatorname{div}\left(\vec{q}_{\text {Darcy }}\right) & -\beta_d \alpha\left(\frac{D^s P_{\text {total }}}{D t}-\frac{1}{B} \frac{D^f P_{\text {fluid }}}{D t}\right) \\ &-\frac{P_{\text {total }}-P_{\text {fluid }}}{(1-\phi) \eta_\phi}=0\end{aligned}
\end{equation}

\noindent \textbf{Momentum conservation for bulk material:}
\begin{equation}
\begin{aligned}\frac{\partial \tau_{i j}}{\partial x_j}-\frac{\partial P_{\text {total }}}{\partial x_i}&=(1-\phi) \rho_{\text {solid }}\left(\frac{D^s V_{i(\text { solid })}}{D t}-g_i\right) \\ &+\phi \rho_{\text {fluid }}\left(\frac{D^f V_{i(\text { fluid })}}{D t}-g_i\right)\end{aligned}
\end{equation}

\noindent \textbf{Momentum conservation for fluid:}
\begin{equation}
\begin{aligned}
& q_{i(\text { Darcy })}=-\frac{k}{\eta_{\text {fluid }}} \\
&\times\left(\frac{\partial P_{\text {fluid }}}{\partial x_i}-\rho_{\text {fluid }} g_i+\rho_{\text {fluid }} \frac{D^f V_{i(\text { fluid })}}{D t}\right)
\end{aligned}
\end{equation}
For a detailed discussion of each term and its physical interpretation, please refer to the original publication~\cite{Gerya2024}.

\subsection{Discretization Strategy}
To accurately solve the equations in section~\ref{ch:governing_equations} and minimize numerical diffusion, a hybrid Eulerian-Lagrangian advection scheme known as the \textit{marker-in-cell} method is employed. This technique uses Lagrangian markers (particles) to represent physical quantities that need to be advected: while these markers move through a fixed Eulerian grid, they carry properties such as mass, velocity, and rock composition. At each timestep, these quantities are interpolated from the markers to the staggered nodal points of the Eulerian grid and vice versa, using a bilinear interpolation method. As the simulation progresses, markers are advected by the local velocity field, and the Eulerian grid solves the governing equations. This hybrid approach reduces numerical diffusion and improves the simulation's accuracy. Additionally, the marker-in-cell method effectively handles complex boundary conditions and heterogeneous material properties, common in geophysical simulations. A more thorough discussion of this method is available in~\cite{Gerya_chapter8}. In the case of this specific experiment, each timestep iteration requires the solution of at least one sparse linear system in order to account for hydro-mechanical coupling. We will use these linear systems as benchmark. All details concerning  the implementation of this specific experiment are available in~\cite{zilio_hydro_mechanical, zilio_subduction_earthquake}.

\subsection{Structure and Conditioning of the Problem}
\label{ch:structure_and_conditioning}
This section details the structure of the linear systems solved as benchmark. In total, 35 iterations are solved. The first 5 are considered stabilization timesteps as the linear system representing the initial conditions is particularly difficult to solve and only some numerical libraries yield correct solutions. During the stabilization period, the computed solution is evaluated and then discarded in favour of the one computed by Eigen's SparseLU algorithm, which produced reliable results. Moreover, no runtimes are recorded and only the remaining 30 timesteps are used for performance evaluation. 

\subsubsection{Sparsity Structure of the System}
The discretization scheme described in the previous sections gives rise to a linear systems of size $360960\times 360960$ with $2575716$ non-zero elements. The sparsity structure is identical between benchmark iterations, so it is possible to always reuse the symbolic factorization computed for the first benchmark iteration. The only exception is the initial-conditions system, which has a slightly different structures and features $2575711$ non-zero elements.
Figure~\ref{fig:sparsity_structure} showcases the sparsity structure of the benchmark matrices, which display a tri-banded structure. Note that the visualization is not to scale with respect to the true size of the matrix.

\begin{figure*}[h!]
    \centering
    \includegraphics[width=0.8\textwidth]{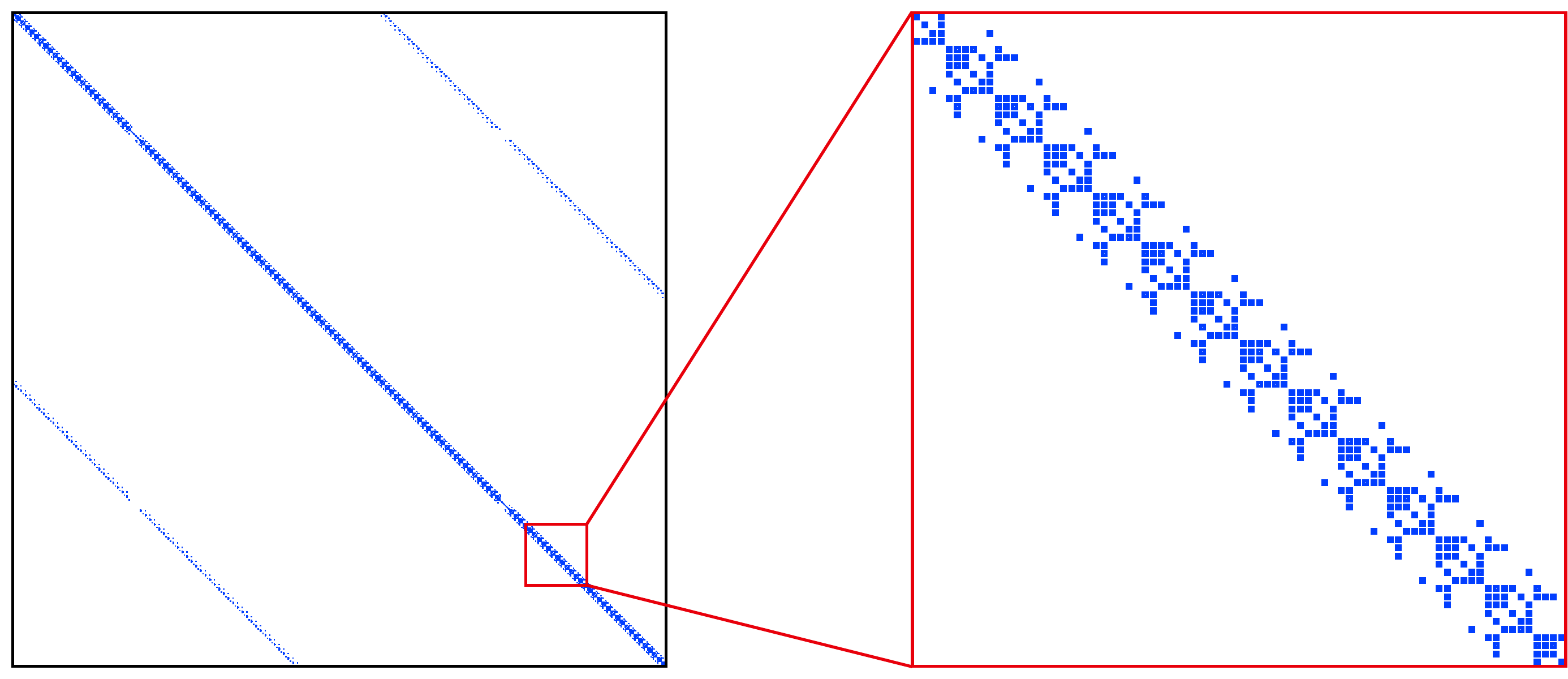}
    \caption{Visualisation of the sparsity structure of the matrix solved during the benchmark iterations. The image on the left shows a section of the matrix detailing the overall tri-banded structure, while the image on the right provides a detailed view of the main diagonal band.}
    \label{fig:sparsity_structure}
    \vspace{0.5cm}
    \includegraphics[width=\textwidth]{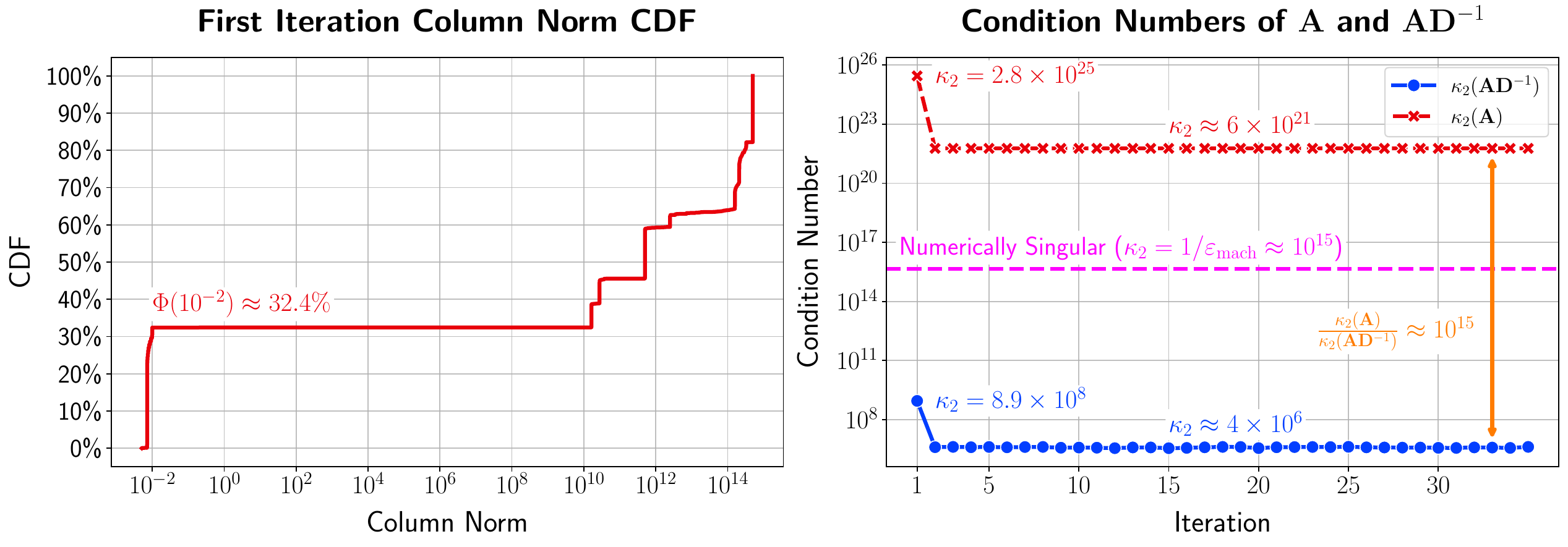}
    \caption{\textbf{Left} – Cumulative Distribution Function (CDF) $\mathbf{\Phi}$ describing the distribution of the column norms of the first benchmark iteration. $32.4\%$ of the columns have norm smaller than $10^{-2}$, while the remaining $67.6\%$ greater than $10^{10}$.}
    \label{fig:col_norm_cdf}
    
    \caption{\textbf{Right} – Condition number for original matrix $\mathbf{A}$ and right-preconditioned matrix $\mathbf{AD}^{-1}$ for every iteration. The original SLSs have conditioning $\kappa_2(\mathbf{A})\geq 1/\varepsilon_{\mathrm{mach}}$ meaning that they are numerically singular according to Theorem \ref{th:numerical_singularity}. The preconditioned systems have condition numbers $\kappa_2(\mathbf{AD}^{-1})\approx 10^6$ meaning that, although they are severely ill-conditioned, they are not numerically singular. Simple column scaling was enough to improve the condition number by a factor of $10^{15}$.}
    \label{fig:iteration_condition_numbers}
\end{figure*}

\subsubsection{Condition Number and Preconditioning}
The benchmark linear systems present a significant challenge due to the highly imbalanced column norms of the coefficients matrix. As detailed in Figure~\ref{fig:col_norm_cdf}, approximately $33\%$ of the columns have norms smaller than $10^{-2}$, while the remaining $67\%$ exhibit norms greater than $10^{10}$. This extreme imbalance leads to an extremely ill-conditioned matrix, with condition number in the order of $10^{21}$, making it numerically singular according to Theorem~\ref{th:numerical_singularity}. A detailed analysis using the rank-revealing feature of the SPQR sparse QR factorization indicates that the matrix has an effective rank of approximately $\operatorname{rank}(\mathbf{A})=243960$, which corresponds precisely to $67\%$ of the total number of columns. 
This result suggests that the $33\%$ of columns with small norms contribute very little to the solution space and instead lie numerically in the kernel of the matrix. Column scaling is thus the natural preconditioning strategy for this problem. Figure \ref{fig:iteration_condition_numbers} shows that the condition number of the preconditioned linear systems drops significantly from the initial $10^{21}$ to around $10^{6}$. Although the system remains ill-conditioned, it is no longer numerically singular, as confirmed by the SPQR QR factorization, which shows that the preconditioned matrices are full rank. The improvement provided by simple column scaling is substantial, reducing the condition number by approximately 15 orders of magnitude. Nevertheless, the initial conditions SLSs remains particularly challenging and some solvers still yield unreliable solutions. 

\section{Benchmark Results}
In this section, we present the results of the benchmark experiments conducted to evaluate the performance and accuracy of various sparse linear solvers. The results are organized into two subsections. We begin by detailing the experimental setup, describing the software and hardware configurations used for the benchmarks to ensure that the results are reproducible and relevant for practical applications. Next, we summarize the key findings from the experiments, comparing the accuracy and performance of the different solvers. This analysis provides a comprehensive evaluation of the solvers, focusing on both their accuracy and overall computational efficiency.

\subsection{Experimental Setup}
The benchmarks were conducted on dedicated nodes equipped with high-performance hardware. Each node is powered by dual AMD EPYC 7773X 64-Core CPUs, providing a total of 128 CPU cores, and 1TB of memory. Additionally, each node is equipped with four NVIDIA A100 80GB PCIe GPUs, which were utilized in the GPU-specific performance tests. All solver libraries were wrapped and invoked from C++ using the Eigen library, which provides a flexible and efficient interface for linear algebra operations. The codebase was compiled using GCC 13.1.0, with maximum safe math optimizations enabled to ensure both performance and numerical accuracy. The tested libraries were compiled using the Spack package manager~\cite{spack} to guarantee a consistent and reproducible build environment. The only exception is PaStiX, which was compiled manually. Multithreading was handled differently depending on the requirements of the library in use, either with OpenMP, Intel TBB, or POSIX Threads.

\subsection{Performance and Accuracy Results}
In this section, we present the performance and accuracy results obtained from testing a total of 16 solvers across 8 different libraries. The solvers were divided into two categories: 5 iterative solvers and 11 direct solvers (9 LU-based and 2 QR-based).

\subsubsection{Iterative Solvers} The benchmark problem posed a significant challenge for iterative solvers due to its extremely ill-conditioned nature. Each solver was configured with a residual tolerance of $10^{-7}$ and a maximum of 1000 iterations. We tested the solvers with no preconditioner as well as with diagonal and incomplete LU preconditioners, after applying column scaling.

\input{tables/iterative_solvers}

Unfortunately, none of the iterative solvers were able to produce correct results, even with the most sophisticated preconditioning techniques. Despite their superior parallel performance, this advantage is moot if the solvers fail to converge to an accurate solution. The inability to handle the problem effectively underscores the difficulty of the benchmark. Table \ref{tab:iterative_solver_accuracy} reports the error for the first benchmark iteration as well as for the initial conditions system, highlighting the shortcomings of iterative methods in this context.

\subsubsection{Direct Solvers}
In contrast to iterative methods, all direct solvers provided reliable solutions across the benchmark iterations, despite the challenging nature of the problem. However, the initial conditions system proved particularly difficult, and some solvers struggled to produce accurate solutions under these conditions. Table \ref{tab:direct_full_benchmark} presents the performance and accuracy results for the full benchmark. These results include both symbolic and numerical factorizations, which were recomputed at each iteration in order to test the performance capabilities of each solver under full workload conditions.

The results presented in Table \ref{tab:direct_full_benchmark} provide an overview of the performance and accuracy of the evaluated direct solvers. The table includes the optimal core count, total time to solution across 30 benchmark iterations, and the median time per iteration, which includes the time for symbolic factorization, numerical factorization, and solution. Additionally, it reports lower and upper error bounds according to Theorem \ref{th:error_x}, tight lower and upper error bounds according to Theorem \ref{th:err_hatx}, along with a tight upper error bound for the initial conditions linear system. Bootstrapped $95\%$ confidence intervals are shown for time-to-solution and median iteration runtime when these intervals exceed $1\%$. The reported relative error bounds correspond to the critical iteration, defined as the iteration with the highest tight upper error bound. In terms of performance, oneMKL PARDISO had the best time to solution and median time per iteration, followed closely by UMFPACK, KLU, and MUMPS. These solvers were the fastest overall, making them well-suited for performance-critical applications. On the accuracy front, SuperLU and SuperLU-dist delivered the most precise results for the benchmark iterations. However, for the initial conditions system, UMFPACK provided the most accurate solution. It’s worth noting that only 8 out of the 11 direct solvers managed to solve the initial conditions linear system, highlighting the difficulty of this particular problem. Generally, LU-based solvers were faster than QR-based solvers: this was expected given the added complexity of computing an orthogonal matrix $\mathbf{Q}$. However, this was not always the case: cuSolver's LU implementation, for example, was about $2.5 \times$ slower than SuiteSparse's SPQR solver, indicating significant variability in performance depending on the quality of an implementation. None of the solvers were able to fully utilize the 128 CPU cores available. The best parallel performance was seen with Intel oneMKL PARDISO, which scaled up to 32 CPU cores, but even then, it only achieved a performance increase of about $2\times$ compared to using a single core. All other solvers showed even less improvement, with some experiencing diminished performance as the core count increased. The poor parallel performance can be attributed to several factors, including the size of the matrix, its sparsity structure, the number of non-zero elements, and the specific implementation details of each solver. It is likely that other problem structures or sizes could benefit more from parallelism than the specific problem benchmarked in this work. For GPU performance, the cuSolver QR GPU implementation was about $5.5 \times$ faster than its CPU counterpart, yet it still lagged behind SuiteSparse's SPQR solver running on 16 CPU cores. Interestingly, GPU support in SuperLU-dist did not result in faster performance and was slightly slower than the CPU-only run, indicating that GPU acceleration does not always yield better results.

As discussed in Section \ref{ch:structure_and_conditioning}, the sparsity pattern of the benchmark matrices remains consistent across iterations. This allows for the re-use of the symbolic factorization computed during the first iteration for all subsequent iterations, potentially reducing computational overhead.
While the cost of symbolic factorization is generally lower than that of numerical factorization given that it has lower asymptotic complexity \cite{survey_direct_solvers}, the extent of this cost depends on the heuristics implemented by the specific solver. To evaluate the impact of re-using symbolic factorization, we reran the benchmarks with this optimization. Most solvers did not show a significant performance boost, with speedups of less than 10\%. Interestingly however, three solvers stood out with substantial improvements:
\begin{itemize}
    \item oneMKL PARDISO became approximately $7 \times$ faster.
    \item MUMPS saw a speedup of about $3.5 \times$.
    \item PaStiX achieved a $2 \times$ speedup.
\end{itemize}

These results make Intel oneMKL PARDISO and MUMPS particularly attractive, as they not only provided a solid performance in the full benchmark but also significantly benefited from the re-use of the symbolic factorization. Table \ref{tab:direct_fact_benchmark} presents the detailed performance and accuracy data for the solvers that exhibited more than a 10\% improvement when re-using symbolic factorization.

\include{tables/benchmarks_table}

\section{Conclusions}
The primary goal of this work was to develop the necessary tools to enable an accurate and sound comparison between different solvers for sparse linear systems, particularly when dealing with severely ill-conditioned problems. In doing so, we aimed to create a resource that serves as a guide for domain scientists and researchers facing the challenges of solving such systems. Additionally, this work provides not only benchmark data but also theoretical insights and a comprehensive list of solvers that can be integrated into production code.

\subsection{Summary of Contributions}
We began by discussing the two main families of solvers – iterative and direct — highlighting their strengths, weaknesses and differences. We introduced the concept of the condition number and its critical role in determining the accuracy of solutions to linear systems. Specifically, we derived both loose and tight relative error bounds for $\|\mathbf{\hat{x}} - \mathbf{x}\| / \|\mathbf{x}\|$ and $\|\mathbf{\hat{x}} - \mathbf{x}\| / \|\mathbf{\hat{x}}\|$ in terms of the matrix condition number $\kappa_2(\mathbf{A})$, and we proved the implications of these bounds on the accuracy of solutions. We emphasized the significance of the condition number in identifying numerical singularity, stating that a matrix with a condition number equal to or greater than $1/\epsilon_{\mathrm{mach}}$ should be considered numerically singular. To address the challenge of computing accurate error bounds, we proposed the Projected Adam method — a novel iterative projected gradient descent algorithm accelerated by momentum. This method offers a simple and efficient way to compute the $\mathcal{L}_2$ condition number of a matrix without resorting to eigenvalues or singular values, requiring only one LU factorization. Finally, we presented benchmark results that incorporated these tools to compare 16 solvers across 11 state-of-the-art numerical libraries. Our results indicated that the best-performing solvers for our benchmark problem were Intel oneAPI MKL PARDISO, UMFPACK, and MUMPS, making them particularly well-suited for severely ill-conditioned problems.

\subsection{Impact and Applications}
The findings from this work have significant implications for real-world applications where solving ill-conditioned sparse linear systems is common. The detailed comparison of solvers provides valuable guidance for selecting the most appropriate solver for specific applications, taking into account both performance and accuracy. This is particularly relevant in fields such as computational physics, engineering and fluid dynamics, where the reliability and efficiency of numerical libraries can greatly influence the outcomes of large-scale simulations. This is especially true in the domain of computational geophysics, where severely ill-posed problems are encountered frequently.

\subsection{Limitations and Future Work}
While this study offers a comprehensive analysis of current solvers, it is important to acknowledge certain limitations. The benchmark problems and matrix sizes used in this study represent specific scenarios, and the results may vary with different sparsity structures, problem types or larger systems. Additionally, the limited scalability of the solvers on multi-core and GPU architectures suggests that the sparsity structure of the benchmark problem was particularly challenging, that the problem size was too small to benefit from parallel resources, or more likely a combination of the two. Future work could explore more advanced preconditioning techniques to improve the robustness of iterative solvers. Moreover, extending the benchmarks to include a broader range of problem sizes and matrix structures would also provide a more comprehensive assessment of true solver capabilities.

\subsection{Final Takeaways}

In conclusion, this work highlights the importance of choosing the right solver based on the problem's condition number and the available computational resources. The tools and insights provided here serve as a valuable resource for domain scientists and researchers, enabling them to soundly assess solutions to ill-conditioned sparse linear systems as well as compare the accuracy and performance of different numerical libraries. As computational challenges continue to evolve, the ongoing development and refinement of numerical solvers will remain crucial in addressing the complexities of modern scientific and engineering problems.

\section{Acknowledgements}
The author would like to express deep gratitude to Professor Taras Gerya (ETH Zurich) for his invaluable insights and discussions during the development of this work.


\bibliographystyle{IEEEbib-No-Dashed-Authors}
\bibliography{main}

\end{document}

%% file: tables/solver_table.tex
\onecolumn
\setlength{\extrarowheight}{15pt} 
\begin{table}[ht]
\centering
\begin{tabular}{lllll}
\toprule
\textbf{Solver}                               & \textbf{Platform} & \parbox{2.3cm}{\raggedright \textbf{Parallelisation Strategy}} & \textbf{Version} & \textbf{Algorithm}\\
\midrule
\textbf{cuSolver LU}~\cite{cuSolver}          & CPU               & No                                                             & 12.4             & \parbox{7.8cm}{\raggedright Not explicitly documented, however very likely sparse LU factorisation via Gaussian elimination with partial pivoting. ~\cite{cuSolverLU1, cuSolverLU2}.}                                               \\
\textbf{cuSolver QR}~\cite{cuSolver}          & CPU, GPU          & \parbox{2.3cm}{\raggedright Single GPU}                        & 12.4             & \parbox{7.8cm}{\raggedright Sparse QR factorization optimized for GPUs~\cite{sparse_QR}.}                                                                              \\
\textbf{Eigen BiCGSTAB}~\cite{eigenweb}       & CPU               & MT                                                             & 3.4.0            & \parbox{7.8cm}{\raggedright Iterative solver with improved stability derived from the BiCG method~\cite{bicgstab}.}                                                                                                \\
\textbf{Eigen DGMRES}~\cite{eigenweb}         & CPU               & No                                                             & 3.4.0            & \parbox{7.8cm}{\raggedright A restarted GMRES solver with deflation. Uses a few approximate eigenvectors at each iteration to build a preconditioner.\cite{DGMRES, performance_DGMRES, restarted_DGMRES}.}                                \\
\textbf{Eigen GMRES}~\cite{eigenweb}          & CPU               & No                                                             & 3.4.0            & \parbox{7.8cm}{\raggedright Iterative solver implementing standard GMRES algorithm. The solution is approximated as the minimal residual vector in a Krylov subspace~\cite{gmres}.}                                                      \\
\textbf{Eigen IDR($s$)}~\cite{eigenweb}           & CPU               & MT                                                             & 3.4.0            & \parbox{7.8cm}{\raggedright Iterative short-recurrence Krylov method based on the Induced Dimension Reduction theorem ~\cite{IDRS}.}                                                                                          \\
\textbf{Eigen IDR($s$)STAB($\ell$)}~\cite{eigenweb}       & CPU               & MT                                                             & 3.4.0            & \parbox{7.8cm}{\raggedright Iterative short-recurrence Krylov method derived as a combination of IDR($s$) and BiCGSTAB($\ell$)~\cite{IDRSTAB, precond_IDRSTAB}.}                                                                                  \\
\textbf{Eigen SparseLU}~\cite{eigenweb}       & CPU               & No                                                             & 3.4.0            & \parbox{7.8cm}{\raggedright Sparse LU factorisation via Gaussian elimination based on the implementation of SupeLU~\cite{eigen_sparselu}.}                                                                                          \\
\textbf{KLU}~\cite{KLU}                       & CPU               & No                                                             & 7.3.1            & \parbox{7.8cm}{\raggedright Fill-in minimising left-looking Sparse LU factorization optimized for circuit simulation and general sparse matrices~\cite{KLU}.}                                                                       \\
\textbf{MUMPS}~\cite{MUMPS}                   & CPU               & \parbox{2.3cm}{\raggedright MT, MPI}                           & 5.6.2            & \parbox{7.8cm}{\raggedright Parallel direct solver using a multifrontal approach~\cite{MUMPS}.}                                                                                                                                     \\
\textbf{oneMKL PARDISO}~\cite{oneMKL_pardiso} & CPU               & \parbox{2.3cm}{\raggedright MT, MPI}                           & 2024.0.0         & \parbox{7.8cm}{\raggedright Parallel supernodal left-right looking sparse LU factorization~\cite{schenk_pardiso}.}                                                                                           \\
\textbf{PaStiX}~\cite{pastix_lib} & CPU, GPU & \parbox{2.3cm}{\raggedright MT, MPI, Multi-GPU} & 5.2.3 & \parbox{7.8cm}{\raggedright Parallel supernodal sparse LU factorization with support for block low-rank (BLR) compression~\cite{pastix_blr, pastix_reord}.}
\\
\textbf{SuperLU}~\cite{superlu_seq}           & CPU               & \parbox{2.3cm}{\raggedright BLAS/LAPACK parallelism}           & 5.3.0            & \parbox{7.8cm}{\raggedright Gaussian elimination with partial pivoting~\cite{superlu_guide, superlu_algs}.}                                                            \\
\textbf{SuperLU-dist}~\cite{superlu_dist}     & CPU, GPU          & \parbox{2.3cm}{\raggedright MT, MPI, Multi-GPU}                & 8.2.1            & \parbox{7.8cm}{\raggedright Distributed memory extension of SuperLU. Uses static pivoting instead of partial pivoting to improve parallelism and reduce communication~\cite{superlu_guide, superlu_algs, superlu_static_pivoting}.} \\
\textbf{SPQR}~\cite{SPQR}                     & CPU, GPU          & \parbox{2.3cm}{\raggedright MT, Single-GPU}                    & 7.3.1            & \parbox{7.8cm}{\raggedright Parallel multifrontal sparse QR factorization.~\cite{SPQR}.}                                                                                  \\
\textbf{UmfpackLU}~\cite{UMFPACK}             & CPU               & No                                                             & 7.3.1            & \parbox{7.8cm}{\raggedright Sparse LU factorization using the unsymmetric multifrontal method~\cite{UMFPACK}.}                                                                                       \\
\bottomrule
\end{tabular}
\caption{Overview of evaluated sparse linear solvers with target platforms, parallelisation strategies, versions, and algorithms. ``MT" stands for multi-threaded, indicating OpenMP~\cite{openmp}, Intel TBB~\cite{intel-tbb} or POSIX Threads support. Whenever a solver is part of a larger library, the version of the library is listed.}
\label{tb:solver_overview}
\end{table}
\twocolumn

%% file: tables/iterative_solvers.tex
\setlength{\extrarowheight}{5pt} 
\begin{table}[ht]
    \centering
    \begin{tabular}{lll}
    \toprule
    \textbf{Solver} & $\kappa(\mathbf{A})\frac{\|\mathbf{A \hat x} - \mathbf{b}\|}{\|\mathbf{A}\| \|\mathbf{\hat x}\|}$ & $\kappa(\mathbf{A_0})\frac{\|\mathbf{A_0 \hat x_0} - \mathbf{b_0}\|}{\|\mathbf{A_0}\| \|\mathbf{\hat x_0}\|}$ \\
    \midrule
    \textbf{BICGSTAB} & nan & 2.9e+10 \\
    \textbf{DGMRES} & 4.2e+02 & 2.9e+10 \\
    \textbf{GMRES} & 4.1e+02 & 1.2e+08 \\
    \textbf{IDR($s$)} & 4.9e+04 & 3.8e+10 \\
    \textbf{IDR($s$)TAB($\ell$)} & 2.6e+02 & 9.3e+07 \\
    \bottomrule
    \end{tabular}
    \caption{Accuracy results of different iterative solvers. The column $\kappa(\mathbf{A})\frac{\|\mathbf{A \hat x} - \mathbf{b}\|}{\|\mathbf{A}\| \|\mathbf{\hat x}\|}$ corresponds to the tight upper bound on the relative error of the first benchmark iteration, while $\kappa(\mathbf{A_0})\frac{\|\mathbf{A_0 \hat x_0} - \mathbf{b_0}\|}{\|\mathbf{A_0}\| \|\mathbf{\hat x_0}\|}$ refers to the tight upper bound on the relative error when solving the initial conditions system.}
    \label{tab:iterative_solver_accuracy}
\end{table}

%% file: tables/benchmarks_table.tex
\onecolumn
\setlength{\extrarowheight}{9pt} 

\begin{sidewaystable}[ht]
\centering

\begin{tabular}{lrlllllll}
\toprule
\textbf{Solver} & \textbf{Cores} & \textbf{Time to Solution} $(s)$ & \textbf{Median Time} $(s)$ & $\frac{1}{\kappa(\mathbf{A})} \frac{\|\mathbf{A \hat x} - \mathbf{b} \|}{\|\mathbf{b}\|}$ & $\kappa(\mathbf{A}) \frac{\|\mathbf{A \hat x} - \mathbf{b} \|}{\|\mathbf{b}\|}$ & $\frac{\|\mathbf{A \hat x} - \mathbf{b}\|}{\|\mathbf{A}\| \|\mathbf{\hat x}\|}$ & $\kappa(\mathbf{A})\frac{\|\mathbf{A \hat x} - \mathbf{b}\|}{\|\mathbf{A}\| \|\mathbf{\hat x}\|}$ & $\kappa(\mathbf{A_0})\frac{\|\mathbf{A_0 \hat x_0} - \mathbf{b_0}\|}{\|\mathbf{A_0}\| \|\mathbf{\hat x_0}\|}$ \\
\midrule
\textbf{cuSolver LU} & $1^*$ & $698.16$ & $23.27$ & 4.0e-35 & 1.4e+09 & 6.8e-33 & 4.0e-11 \cmark & 2.4e-07 \cmark \\
\textbf{cuSolver QR} & $1^*$ & $4683.35$ & $156.11$ & 6.6e-35 & 2.3e+09 & 1.1e-32 & 6.6e-11 \cmark & 6.6e-07 \cmark \\
\textbf{cuSolver QR GPU} & $1^*$ & $825.35$ & $27.50$ & 3.3e-35 & 1.2e+09 & 5.7e-33 & 3.3e-11 \cmark & 1.4e-07 \cmark \\
\textbf{Eigen SparseLU} & $1^*$ & $154.15$ & $5.14$ & 8.6e-35 & 3.0e+09 & 1.5e-32 & 8.7e-11 \cmark & 4.1e-07 \cmark \\
\textbf{KLU} & $1^*$ & $87.05$ & $2.89$ & 1.3e-35 & 4.4e+08 & 2.2e-33 & 1.3e-11 \cmark & 1.1e-04 \xmark \\
\textbf{MUMPS} & 1 & $94.26$ & $3.14$ & 2.4e-35 & 8.2e+08 & 4.0e-33 & 2.4e-11 \cmark & 1.1e-06 \xmark \\
\textbf{oneMKL PARDISO} & 32 & $\mathbf{57.03}$ & $\mathbf{1.90^{ \ 1.00\%}}$ & 8.1e-31 & 2.8e+13 & 1.4e-28 & 8.2e-07 \cmark & 5.1e-02 \xmark \\
\textbf{PaStiX} & $1^*$ & $123.29^{ \ 3.02\%}_{ \ 2.13\%}$ & $4.00^{ \ 2.98\%}_{ \ 1.89\%}$ & 4.9e-34 & 1.7e+10 & 8.4e-32 & 5.0e-10 \cmark & 2.5e+00 \xmark \\
\textbf{SPQR} & 16 & $283.88_{ \ 1.08\%}$ & $9.53_{ \ 2.10\%}$ & 3.4e-35 & 1.2e+09 & 5.8e-33 & 3.4e-11 \cmark & 3.4e-04 \xmark \\
\textbf{SuperLU} & 1 & $150.65$ & $5.01$ & 5.1e-35 & 1.8e+09 & 8.8e-33 & 5.2e-11 \cmark & 3.4e-07 \cmark \\
\textbf{SuperLU dist} & 4 & $134.75$ & $4.47$ & 7.2e-37 & 2.5e+07 & 1.3e-33 & 7.6e-12 \cmark & 4.4e-08 \cmark \\
\textbf{SuperLU dist GPU} & 8 & $155.29^{ \ 1.59\%}_{ \ 1.45\%}$ & $5.10^{ \ 4.68\%}_{ \ 1.43\%}$ & \textbf{7.2e-37} & \textbf{2.5e+07} & \textbf{1.3e-33} & \textbf{7.6e-12 \cmark} & 4.4e-08 \cmark \\
\textbf{UMFPACK} & 1 & $65.87$ & $2.20$ & 7.7e-36 & 2.7e+08 & 1.3e-33 & 7.8e-12 \cmark & \textbf{7.7e-12 \cmark} \\
\bottomrule
\end{tabular}
\caption{Performance and accuracy results of evaluated sparse direct solvers. Each row presents the following data: optimal core count (* indicates no multithreading support or compiled without multithreading), total time to compute 30 benchmark iterations (including both symbolic and numerical factorization for each iteration), median iteration time, worst-case error for a benchmark linear system $\mathbf{Ax} = \mathbf{B}$, and tight error upper bound for the initial conditions system $\mathbf{A_0x_0} = \mathbf{B_0}$. Bootstrapped $95\%$ relative CIs are provided for total time and median iteration time when exceeding $1\%$. The values in bold represent the best recorded result.}
\label{tab:direct_full_benchmark}
\vspace{0.5cm}

\begin{tabular}{lrlllllll}
\toprule
\textbf{Solver} & \textbf{Cores} & \textbf{Time to Solution} $(s)$ & \textbf{Median Time} $(s)$ & $\frac{1}{\kappa(\mathbf{A})} \frac{\|\mathbf{A \hat x} - \mathbf{b} \|}{\|\mathbf{b}\|}$ & $\kappa(\mathbf{A}) \frac{\|\mathbf{A \hat x} - \mathbf{b} \|}{\|\mathbf{b}\|}$ & $\frac{\|\mathbf{A \hat x} - \mathbf{b}\|}{\|\mathbf{A}\| \|\mathbf{\hat x}\|}$ & $\kappa(\mathbf{A})\frac{\|\mathbf{A \hat x} - \mathbf{b}\|}{\|\mathbf{A}\| \|\mathbf{\hat x}\|}$ & $\kappa(\mathbf{A_0})\frac{\|\mathbf{A_0 \hat x_0} - \mathbf{b_0}\|}{\|\mathbf{A_0}\| \|\mathbf{\hat x_0}\|}$ \\
\midrule
\textbf{MUMPS} & 1 & $26.92$ & $0.90^{ \ 1.28\%}$ & \textbf{2.3e-35} & \textbf{8.1e+08} & \textbf{4.0e-33} & \textbf{2.3e-11 \cmark} & 1.1e-06 \xmark \\
\textbf{oneMKL PARDISO} & 32 & $\mathbf{8.09^{ \ 4.29\%}_{ \ 3.34\%}}$ & $\mathbf{0.26^{ \ 4.60\%}_{ \ 3.33\%}}$ & 5.8e-31 & 2.0e+13 & 9.9e-29 & 5.8e-07 \cmark & 5.1e-02 \xmark \\
\textbf{PaStiX} & 1 & $64.35^{ \ 6.76\%}_{ \ 4.86\%}$ & $1.98^{ \ 8.01\%}_{ \ 4.55\%}$ & 5.5e-34 & 1.9e+10 & 9.3e-32 & 5.5e-10 \cmark & 2.5e+00 \xmark \\
\bottomrule
\end{tabular}
\caption{Performance results with reused symbolic factorization. This table lists only the solvers that showed more than a $10\%$ performance improvement by reusing the symbolic factorization from the first benchmark iteration. Column definitions are the same as in table~\ref{tab:direct_full_benchmark}.}
\label{tab:direct_fact_benchmark}

\end{sidewaystable}

\twocolumn